\documentclass[a4paper,11pt,fleqn]{article}
\usepackage{amsmath}
\usepackage{amssymb}
\usepackage{theorem}
\usepackage{pstricks}
\usepackage{euscript}
\usepackage{epic,eepic}
\usepackage{graphicx}
\PassOptionsToPackage{normalem}{ulem}
\newcommand{\menge}[2]{\big\{{#1} \mid {#2}\big\}} 

\newcommand{\psscal}[2]{\langle\langle\langle{#1}\mid{#2}%
\rangle\rangle\rangle}

\newcommand{\emp}{\ensuremath{{\varnothing}}}

\newcommand{\infconv}{\ensuremath{\mbox{\small$\,\square\,$}}}
\newcommand{\scal}[2]{\left\langle{#1}\mid {#2} \right\rangle}

\newcommand{\vuo}{\ensuremath{\mbox{\footnotesize$\square$}}}

\newcommand{\HH}{\ensuremath{\mathcal H}}
\newcommand{\GG}{\ensuremath{\mathcal G}}

\newcommand{\BL}{\ensuremath{\EuScript B}\,}

\newcommand{\BP}{\ensuremath{\EuScript P}}

\newcommand{\KKK}{\ensuremath{\boldsymbol{\mathcal K}}}

\newcommand{\RP}{\ensuremath{\left[0,+\infty\right[}}
\newcommand{\RPP}{\ensuremath{\,\left]0,+\infty\right[}}

\newcommand{\NN}{\ensuremath{\mathbb N}}
\newcommand{\dom}{\ensuremath{\operatorname{dom}}}

\newcommand{\inte}{\ensuremath{\operatorname{int}}}

\newcommand{\ran}{\ensuremath{\operatorname{ran}}}
\newcommand{\zer}{\ensuremath{\operatorname{zer}}}
\newcommand{\gra}{\ensuremath{\operatorname{gra}}}

\newcommand{\vv}{\ensuremath{\boldsymbol{v}}}
\newcommand{\uu}{\ensuremath{\boldsymbol{u}}}
\newcommand{\zz}{\ensuremath{\boldsymbol{z}}}
\newcommand{\pp}{\ensuremath{\boldsymbol{p}}}
\newcommand{\qq}{\ensuremath{\boldsymbol{q}}}
\newcommand{\xx}{\ensuremath{\boldsymbol{x}}}
\newcommand{\yy}{\ensuremath{\boldsymbol{y}}}
\newcommand{\ee}{\ensuremath{\boldsymbol{e}}}
\newcommand{\bb}{\ensuremath{\boldsymbol{b}}}
\newcommand{\cc}{\ensuremath{\boldsymbol{c}}}

\newcommand{\dd}{\ensuremath{\boldsymbol{d}}}

\newcommand{\aaa}{\ensuremath{\boldsymbol{a}}}
\newcommand{\xxx}{\ensuremath{\bar{\boldsymbol{x}}}}

\newcommand{\ww}{\ensuremath{\boldsymbol{w}}}
\newcommand{\BB}{\ensuremath{\boldsymbol{B}}}

\newcommand{\UU}{\ensuremath{\boldsymbol{U}}}
\newcommand{\AAA}{\ensuremath{\boldsymbol{A}}}
\newcommand{\BBB}{\ensuremath{\boldsymbol{B}}}

\newcommand{\Id}{\ensuremath{\operatorname{Id}}}

\newcommand{\IId}{\ensuremath{\boldsymbol{\operatorname{Id}}}}
\newcommand{\weakly}{\ensuremath{\rightharpoonup}}

\newcommand{\pinf}{\ensuremath{+\infty}}
\newtheorem{theorem}{Theorem}[section]

\newtheorem{corollary}[theorem]{Corollary}

\theoremstyle{plain}{\theorembodyfont{\rmfamily}
}
\theoremstyle{plain}{\theorembodyfont{\rmfamily}
}
\theoremstyle{plain}{\theorembodyfont{\rmfamily}
}
\theoremstyle{plain}{\theorembodyfont{\rmfamily}
\newtheorem{example}[theorem]{Example}}
\theoremstyle{plain}{\theorembodyfont{\rmfamily}
\newtheorem{problem}[theorem]{Problem}}
\theoremstyle{plain}{\theorembodyfont{\rmfamily}
\newtheorem{remark}[theorem]{Remark}}
\theoremstyle{plain}{\theorembodyfont{\rmfamily}
}

\numberwithin{equation}{section}

\begin{document}
\title{\sffamily A variable metric extension of the 
forward--backward--forward algorithm for monotone operators
\footnote{This work  was partially supported by Grant
102.01-2012.15  of the Vietnam National Foundation  for Science and
Technology Development (NAFOSTED).
}}
\author{ B$\grave{\text{\u{a}}}$ng C\^ong V\~u\\[5mm]
\small UPMC Universit\'e Paris 06\\
\small Laboratoire Jacques-Louis Lions -- UMR CNRS 7598\\
\small 75005 Paris, France\\[2mm]
\small\url{vu@ljll.math.upmc.fr} \\
}
\date{~}
\maketitle
\begin{abstract}
We propose a  variable metric extension of the
forward--backward-forward algorithm for finding a zero of the sum of
a maximally monotone operator and a monotone Lipschitzian operator
in Hilbert spaces.  In turn, this framework provides a variable
metric splitting algorithm for solving monotone inclusions involving
sums of composite operators. Monotone operator splitting methods 
recently proposed in the literature are recovered as special cases. 
\end{abstract}



\section{Introduction}

A basic problem in applied monotone operator theory is to find
a zero of a maximally monotone operator $A$ on a real Hilbert space $\HH$.
This problem can be solved by
the proximal point algorithm proposed in~\cite{Rock76} which requires only  
the resolvent of $A$, provided it is easy to implement numerically. 
In order to get more efficient proximal algorithms, 
some authors have proposed the use of variable 
metric or preconditioning in such algorithms
\cite{Bonn95,Burke99,Burke00,Guad2012,Lemarechal97a,solodov08,Qi95}.

\noindent
This problem was then extended to the problem of finding a zero of the sum of
a maximally monotone operator $A$ and a cocoercive operator $B$
(i.e., $B^{-1}$ is strongly monotone). In such instances, 
the forward-backward splitting algorithm \cite{plc2010,Siop1,mercier79,Tseng91} 
can be used.
Recently, this algorithm has been investigated in the context of variable metric \cite{Varm12}. 
In the case when $B$ is only Lipschitzian and not cocoercive, 
the problem can be solved by the
forward-backward-forward splitting algorithm \cite{siop2,Tseng00}. 
New applications of this basic algorithm to more complex 
monotone inclusions are presented in \cite{siop2,plc6}.

In the present paper, we propose a variable metric version of 
the forward-backward-forward splitting
algorithm. In
Section \ref{s:nb}, we recall notation and background on convex analysis 
and monotone operator theory. In Section \ref{s:vmfbf},
we present our variable metric forward-backward-forward splitting 
algorithm. In
Section \ref{s:vmpdfbf}, the results of Section \ref{s:vmfbf}
are used to develop a variable metric primal--dual algorithm for solving 
the type of composite inclusions considered in \cite{plc6}. 

\section{Notation and background }
\label{s:nb}
Throughout, $\HH$, $\GG$, and $(\GG_i)_{1\leq i\leq m}$ are real 
Hilbert spaces. Their scalar products and  associated norms are respectively denoted 
by $\scal{\cdot}{\cdot}$ and  $\|\cdot\|$. 
We denote by $\BL(\HH,\GG)$ the space of bounded linear operators 
from $\HH$ to $\GG$. The adjoint of $L\in\BL(\HH,\GG)$ is denoted by $L^*$.
We set $\BL(\HH)=\BL(\HH,\HH)$.
The symbols $\weakly$ and $\to$ denote respectively weak and strong 
convergence, and $\Id$ denotes the identity operator,
 and $B(x;\rho)$ denotes the
closed ball of center $x\in\HH$ and radius $\rho\in\RPP$. 
The interior of $C\subset\HH$ is denoted by $\inte C$.
We denote  
by $\ell_+^1(\NN)$ the set of summable sequences in 
$\RP$.

Let $M_1$ and $M_2$ be self-adjoint operators in $\BL(\HH)$, we write 
$
M_1\succcurlyeq M_2\;\text{if and only if}\;(\forall x\in\HH)\;
\scal{M_1x}{x}\geq\scal{M_2x}{x}.
$
Let $\alpha\in\left]0,+\infty\right[$. We set
\begin{equation}
\BP_{\alpha}(\HH)=\menge{M\in\BL(\HH)}{M^* =M \quad \text{and}\quad M\succcurlyeq\alpha\Id}.
\end{equation}
 Moreover, for every 
$M\in\BP_\alpha(\HH)$, we define respectively a scalar product and a norm by
\begin{equation}
\label{Unorm}
(\forall x\in\HH)(\forall y\in\HH)\quad\scal{x}{y}_M=\scal{Mx}{y}
\quad\text{and}\quad\|x\|_M=\sqrt{\scal{Mx}{x}}.
\end{equation}
Let $A\colon\HH\to 2^{\HH}$ be a set-valued operator.
The domain is $\dom A=\menge{x\in\HH}{Ax\neq\emp}$,
 and the graph of $A$ is 
$\gra A=\menge{(x,u)\in\HH\times\HH}{u\in Ax}$.
The set of zeros 
of $A$ is  $\zer A=\menge{x\in\HH}{0\in Ax}$, and the range of $A$ is
$\ran A=\menge{u\in\HH}{(\exists\; x\in\HH)\;u\in Ax}$. 
The inverse of $A$ is $A^{-1}\colon\HH\mapsto 2^{\HH}\colon u\mapsto 
\menge{x\in\HH}{u\in Ax}$, and the resolvent of $A$ is
\begin{equation}
\label{e:resolvent}
J_A=(\Id+A)^{-1}.
\end{equation}
Moreover, $A$ is monotone if 
\begin{equation}
(\forall(x,y)\in\HH\times\HH)
(\forall(u,v)\in Ax\times Ay)\quad\scal{x-y}{u-v}\geq 0,
\end{equation}
and maximally monotone if it is monotone and there exists no 
monotone operator 
$B\colon\HH\to2^\HH$ such that $\gra A\subset\gra B$ and $A\neq B$.
We say that $A$ is uniformly monotone 
at $x\in\dom A$ if there exists an 
increasing function $\phi_A\colon\left[0,+\infty\right[\to 
\left[0,+\infty\right]$ vanishing only at $0$ such that 
\begin{equation}\label{oioi}
\big(\forall u\in Ax\big)\big(\forall (y,v)\in\gra A\big)
\quad\scal{x-y}{u-v}\geq\phi_A(\|x-y\|).
\end{equation}

\section{Variable metric forward-backward-forward splitting algorithm}
\label{s:vmfbf}
The forward-backward-forward splitting algorithm 
was first proposed in \cite{Tseng00} to solve inclusion involving
the sum of a maximally monotone operator and a monotone Lipschitzian  operator. 
In \cite{siop2}, it was revisited to include computational errors. 
Below, we extend it to a variable metric setting. 
\begin{theorem}
\label{l:Tsenglemma}
Let $\KKK$ be a real Hilbert space with the scalar product $\psscal{\cdot}{\cdot}$
and the associated norm $||||\cdot||||$.
Let $\alpha$ and $\beta$ be in $ \left]0,+\infty\right[$, 
let $(\eta_n)_{n\in\NN}$ be a sequence in $\ell_{+}^1(\NN)$,
and let $(\UU_n)_{n\in\NN}$ 
be a sequence in $\BL(\KKK)$ such that 
\begin{equation}
\label{e:tse}
\mu =\sup_{n\in\NN} \|\UU_n\| < +\infty\quad 
\text{and}\quad (1+\eta_n)\UU_{n+1} \succcurlyeq \UU_n\in\BP_{\alpha}(\KKK).
\end{equation}
Let $\AAA\colon\KKK\to 2^{\KKK}$ be maximally monotone,
let $\BB\colon\KKK\to \KKK$ be a monotone and $\beta$-Lipschitzian operator on
$\KKK$ such that $\zer(\AAA+\BB) \not= \emp$.
Let $(\aaa_n)_{n\in\NN}$, $(\bb_n)_{n\in\NN}$, and $(\cc_n)_{n\in\NN}$ 
be absolutely summable sequences in $\KKK$.
Let $\xx_0\in \KKK$, let $\varepsilon \in \left]0, 1/(\beta\mu +1)\right[$,
let $(\gamma_n)_{n\in\NN}$ be a sequence 
in $\left[\varepsilon, (1-\varepsilon)/(\beta\mu)\right]$, and set
\begin{equation}
\label{e:Tsenga}
(\forall n\in\NN)\quad
\begin{array}{l}
\left\lfloor
\begin{array}{l}
\yy_n= \xx_n-\gamma_n\UU_n(\BB \xx_n+\aaa_n)\\[1mm]
\pp_n = J_{\gamma_n \UU_n \AAA}\yy_n + \bb_n\\
\qq_n = \pp_n -\gamma_n \UU_n(\BB\pp_n+\cc_n)\\
\xx_{n+1} = \xx_n -\yy_n +\qq_n.
\end{array}
\right.\\[2mm]
\end{array}
\end{equation}
Then the following hold for some $\overline{\xx}\in\zer(\AAA+\BB)$.
\begin{enumerate}
 \item
\label{Tsengi} $\sum_{n\in\NN}||||\xx_n-\pp_n||||^2 < \pinf$ and 
$\sum_{n\in\NN}||||\yy_n-\qq_n||||^2 < \pinf$.
\item \label{Tsengii}
$\xx_n \weakly \overline{\xx}$ and 
$\pp_n \weakly \overline{\xx}$.
\item\label{Tsengiii} Suppose that one of the following is satisfied:
\begin{enumerate}
\item \label{Tsengiiia}$\varliminf d_{\zer(\AAA+\BB)}(\xx_n) =0$.
 \item\label{Tsengiiib} $\AAA + \BB$ is demiregular 
(see \cite[Definition 2.3]{plc2010}) at $\overline{\xx}$.
\item \label{Tsengiiic}$\AAA$ or $\BBB$ is uniformly monotone at $\overline{\xx}$.
\item \label{Tsengiiid}$\inte \zer(\AAA+\BB)\neq\emp$ 
and there exists $(\nu_n)_{n\in\NN}\in\ell_{+}^{1}(\NN)$
such that $(\forall n\in\NN)\; (1+\nu_n) \UU_n \succeq \UU_{n+1}$.
\end{enumerate}
Then $\xx_n\to\overline{\xx}$ and $\pp_n\to\overline{\xx}$.
\end{enumerate}
\end{theorem}
\begin{proof}
It follows from \cite[Lemma 3.7]{Varm12} that the
sequences $(\xx_n)_{n\in\NN}$, $(\yy_n)_{n\in\NN}$, 
$(\pp_n)_{n\in\NN}$ and $(\qq_n)_{n\in\NN}$ are well defined. 
Moreover, using \cite[Lemma~2.1(i)(ii)]{Guad2012} and \eqref{e:tse},
we obtain
\begin{equation}
\label{e:absol}
\big(\forall (\boldsymbol{z}_n)_{n\in\NN}\in\KKK^{\NN}\big)
\quad \sum_{n\in\NN}||||\boldsymbol{z}_n |||| < +\infty 
\quad \Leftrightarrow\quad 
\sum_{n\in\NN}||||\boldsymbol{z}_n||||_{\UU_{n}^{-1}} < +\infty
\end{equation}
and 
\begin{equation}
\label{e:absolun}
\big(\forall (\boldsymbol{z}_n)_{n\in\NN}\in\KKK^{\NN}\big)\quad
\sum_{n\in\NN}||||\boldsymbol{z}_n |||| < +\infty 
\quad \Leftrightarrow\quad  
\sum_{n\in\NN}||||\boldsymbol{z}_n||||_{\UU_{n}} < +\infty.
\end{equation}
Let us set 
\begin{equation}
\label{e:Tseng}
(\forall n\in\NN)\quad
 \begin{cases}
 \widetilde{\yy}_n= \xx_n-\gamma_n \UU_n\BB\xx_n\\
\widetilde{\pp}_n = J_{\gamma_n \UU_n \AAA} \widetilde{\yy}_n\\
\widetilde{\qq}_n = \widetilde{\pp}_n -\gamma_n \UU_n\BB\widetilde{\pp}_n\\
\widetilde{\xx}_{n+1} = \xx_n - \widetilde{\yy}_n +\widetilde{\qq}_n,
\end{cases}
\;\; \text{and} \quad
\begin{cases}
\uu_n = \gamma^{-1}_n\UU^{-1}_n(\xx_n-\widetilde{\pp}_n) + \BB\widetilde{\pp}_n-\BB\xx_n\\
\ee_n = \widetilde{\xx}_{n+1} - \xx_{n+1}\\
\dd_n =\qq_n-\widetilde{\qq}_n +  \widetilde{\yy}_n-\yy_{n}.
\end{cases}
\end{equation}
Then \eqref{e:Tseng} yields 
\begin{equation}
 (\forall n\in\NN)\quad \uu_n = \gamma^{-1}_n\UU^{-1}_n(\ \widetilde{\yy}_n-\widetilde{\pp}_n) 
+ \BB\widetilde{\pp}_n \in \AAA\widetilde{\pp}_n+\BB\widetilde{\pp}_n,
\end{equation}
and \eqref{e:Tseng}, \eqref{e:Tsenga}, Lemma \cite[Lemma 3.7(ii)]{Varm12},
and the Lipschitzianity of $\BB$ on $\KKK$ yield 
\begin{equation}
\label{e:abs1}
(\forall n\in\NN)\quad
\begin{cases}
||||\yy_n -\widetilde{\yy}_n||||_{\UU_{n}^{-1}} \leq  (\beta\mu)^{-1}||||\aaa_n ||||_{\UU_{n}}\\
 ||||\pp_n -\widetilde{\pp}_n ||||_{\UU_{n}^{-1}} \leq   
||||\bb_n ||||_{\UU_{n}^{-1}} + (\beta\mu)^{-1} ||||\aaa_n ||||_{\UU_{n}}\\
 ||||\qq_n -\widetilde{\qq}_n ||||_{\UU_{n}^{-1}} \leq 
2\Big( ||||\bb_n ||||_{\UU_{n}^{-1}} + (\beta\mu)^{-1} ||||\aaa_n ||||_{\UU_{n}}\Big)  
+ (\beta\mu)^{-1} ||||\cc_n ||||_{\UU_{n}}.
\end{cases}
\end{equation}
Since $(\aaa_n)_{n\in\NN}$, $(\bb_n)_{n\in\NN}$, and $(\cc_n)_{n\in\NN}$ 
are absolutely summable sequences in $\KKK$, we derive from
\eqref{e:absol}, \eqref{e:absolun}, \eqref{e:Tseng}, and \eqref{e:abs1} that
\begin{alignat}{2}
 \label{e:pabsl}
\begin{cases}
 \sum_{n\in\NN} ||||\pp_n -\widetilde{\pp}_n ||||< +\infty
\quad \text{and} \quad  \sum_{n\in\NN}  ||||\pp_n -\widetilde{\pp}_n ||||_{\UU_{n}^{-1}} < +\infty\\
 \sum_{n\in\NN} ||||\qq_n -\widetilde{\qq}_n ||||< +\infty
\quad \text{and} \quad  \sum_{n\in\NN}  ||||\qq_n -\widetilde{\qq}_n ||||_{\UU_{n}^{-1}} < +\infty\\
 \sum_{n\in\NN} ||||\dd_n ||||< +\infty
\quad \text{and} \quad  \sum_{n\in\NN}  ||||\dd_n ||||_{\UU_{n}^{-1}} < +\infty.
\end{cases}
\end{alignat}
Now, let $\xx\in \zer(\AAA+\BB)$. Then, for every $n\in\NN$,
 $(\xx,-\gamma_n \UU_n\BB\xx) \in\gra (\gamma_n\UU_n\AAA)$
and \eqref{e:Tseng} yields 
$(\widetilde{\pp}_n,\widetilde{\yy}_n-\widetilde{\pp}_n) \in \gra(\gamma_n \UU_n\AAA)$. 
Hence, by  monotonicity of $\UU_n\AAA$ with respect to the scalar product 
$\psscal{\cdot}{\cdot}_{\UU_{n}^{-1}}$, 
we have  
$
\psscal{\widetilde{\pp}_n-\xx}{\widetilde{\pp}_n-\widetilde{\yy}_n-\gamma_n\UU_n\BB\xx}_{\UU_{n}^{-1}} \leq 0.
$
Moreover, by  monotonicity of $\UU_n\BB$ with respect to 
the scalar product $\psscal{\cdot}{\cdot}_{\UU_{n}^{-1}}$, we also have 
$
\psscal{\widetilde{\pp}_n-\xx}{\gamma_n \UU_n \BB\xx -\gamma_n \UU_n \BB\widetilde{\pp}_n }_{\UU_{n}^{-1}}\leq 0.
$
By adding the last two inequalities, we obtain 
\begin{equation}
(\forall n\in\NN)\quad
\psscal{\widetilde{\pp}_n-\xx}{\widetilde{\pp}_n-\widetilde{\yy}_n-\gamma_n \UU_n\BB\widetilde{\pp}_n}_{\UU_{n}^{-1}} \leq 0.
\end{equation}
In turn, we derive from \eqref{e:Tseng} that 
\begin{alignat}{2}
\label{e:tse1}
(\forall n\in\NN)\quad
 &2\gamma_n \psscal{\widetilde{\pp}_n-\xx}{ \UU_n\BB\xx_n-\UU_n\BB\widetilde{\pp}_n}_{\UU_{n}^{-1}}\notag\\
&= 2\psscal{\widetilde{\pp}_n-\xx}{\widetilde{\pp}_n-\widetilde{\yy}_n -\gamma_n \UU_n \BB\widetilde{\pp}_n}_{\UU_{n}^{-1}}
\notag\\
&\quad +2 \psscal{\widetilde{\pp}_n-\xx}{\gamma_n \UU_n \BB\xx_n +\widetilde{\yy}_n-\widetilde{\pp}_n }_{\UU_{n}^{-1}}\notag\\ 
& \leq 2 \psscal{\widetilde{\pp}_n-\xx}{\gamma_n \UU_n \BB\xx_n 
+\widetilde{\yy}_n-\widetilde{\pp}_n }_{\UU_{n}^{-1}}\notag\\
&= 2 \psscal{\widetilde{\pp}_n-\xx}{\xx_n-\widetilde{\pp}_n }_{\UU_{n}^{-1}}\notag\\
& = ||||\xx_n-\xx||||_{\UU^{-1}_n}^2-||||\widetilde{\pp}_n-\xx||||_{\UU_{n}^{-1}}^2-||||\xx_n-\widetilde{\pp}_n|||_{\UU_{n}^{-1}}^2.
\end{alignat}
Hence, using \eqref{e:Tseng}, 
\eqref{e:tse1}, the $\beta$-Lipschitz continuity of $\BB$, \eqref{e:tse}, and 
\cite[Lemma~2.1(ii)]{Guad2012}, for every $n\in\NN$, we obtain 
\begin{alignat}{2}
\label{e:troicho}
||||\widetilde{\xx}_{n+1}-\xx||||_{\UU_{n}^{-1}}^2 &= 
||||\widetilde{\qq}_n + \xx_n -\widetilde{\yy}_n -\xx||||_{\UU_{n}^{-1}}^2\notag\\
& = ||||(\widetilde{\pp}_n-\xx) + \gamma_n \UU_n(\BB\xx_n-\BB\widetilde{\pp}_n)||||_{\UU_{n}^{-1}}^2\notag\\
&= ||||\widetilde{\pp}_n-\xx||||_{\UU_{n}^{-1}}^2 +2\gamma_n\psscal{\widetilde{\pp}_n-\xx}{\BB\xx_n-\BB\widetilde{\pp}_n}\notag\\
&\hspace{5.5cm} + \gamma_{n}^2 ||||\UU_n (\BB\xx_n-\BB\widetilde{\pp}_n)||||_{\UU_{n}^{-1}}^2\notag\\
&\leq ||||\xx_n-\xx||||_{\UU^{-1}_n}^2 -||||\xx_n-\widetilde{\pp}_n||||_{\UU_{n}^{-1}}^2 +\gamma_{n}^2\mu\beta^2
||||\xx_n-\widetilde{\pp}_n||||^2\notag\\  
&\leq ||||\xx_n-\xx||||_{\UU^{-1}_n}^2 - \mu^{-1}||||\xx_n-\widetilde{\pp}_n||||^2 
+\gamma_{n}^2\mu\beta^2||||\xx_n-\widetilde{\pp}_n||||^2.
\end{alignat}
Hence,     
it follows from \eqref{e:tse} and \cite[Lemma~2.1(i)]{Guad2012}
that
\begin{equation}
\label{e:t21ab}
(\forall n\in\NN)\quad
 ||||\widetilde{\xx}_{n+1}-\xx||||_{\UU_{n+1}^{-1}}^2 \leq 
(1+\eta_n)||||\xx_n-\xx||||_{\UU^{-1}_n}^2 
-\mu^{-1}(1-\gamma^{2}_n \beta^2\mu^2)||||\xx_n-\widetilde{\pp}_n||||^2.
\end{equation}
Consequently, 
\begin{equation}
\label{e:t21a}
(\forall n\in\NN)\quad
 ||||\widetilde{\xx}_{n+1}-\xx||||_{\UU_{n+1}^{-1}} \leq (1+\eta_n)||||\xx_n-\xx||||_{\UU^{-1}_n}.
\end{equation}
For every $n\in\NN$, set 
\begin{equation} 
\varepsilon_n = \sqrt{\mu\alpha^{-1}}\Big(
2\big( ||||\bb_n ||||_{\UU_{n}^{-1}} + (\beta\mu)^{-1} ||||\aaa_n ||||_{\UU_{n}}\big)  
+ (\beta\mu)^{-1} ||||\cc_n ||||_{\UU_{n}} + (\beta\mu)^{-1}||||\aaa_n ||||_{\UU_{n}}\Big) .
\end{equation}
Then $(\varepsilon_n)_{n\in\NN}$ is summable by \eqref{e:absol} and \eqref{e:absolun}. We derive 
from \cite[Lemma~2.1(ii)(iii)]{Guad2012}, and \eqref{e:pabsl} that 
\begin{alignat}{2}
(\forall n\in\NN)\quad ||||\ee_n||||_{\UU_{n+1}^{-1}}& 
= ||||\widetilde{\xx}_{n+1}-\xx_{n+1}||||_{\UU_{n+1}^{-1}}\notag\\
&\leq \sqrt{\alpha^{-1}}||||\widetilde{\xx}_{n+1}-\xx_{n+1}||||\notag\\
&\leq \sqrt{\mu\alpha^{-1}}||||\widetilde{\xx}_{n+1}-\xx_{n+1}||||_{\UU_{n}^{-1}}\notag\\
&\leq  \sqrt{\mu\alpha^{-1}}
(||||\widetilde{\yy}_n-\yy_{n}||||_{\UU_{n}^{-1}} + ||||\widetilde{\qq}_{n}-\qq_{n}||||_{\UU_{n}^{-1}})\notag\\
&\leq \varepsilon_n.
\end{alignat}
In turn, we derive from \eqref{e:t21a} that  
\begin{alignat}{2}
\label{e:oac}
(\forall n\in\NN)\quad
||||\xx_{n+1}-\xx||||_{\UU_{n+1}^{-1}} &\leq 
||||\widetilde{\xx}_{n+1}-\xx||||_{\UU_{n+1}^{-1}}+  
||||\widetilde{\xx}_{n+1}-\xx_{n+1}||||_{\UU_{n+1}^{-1}}\notag\\
&\leq||||\widetilde{\xx}_{n+1}-\xx||||_{\UU_{n+1}^{-1}}+ \varepsilon_n\notag\\
&\leq (1+\eta_n)||||\xx_n-\xx||||_{\UU^{-1}_n}+\varepsilon_n.
\end{alignat}
This shows that $(\xx_n)_{n\in\NN}$ is $|\cdot|$--quasi-Fej\'er monotone with respect to the target set
$\zer(\AAA+\BB)$ relative to $(\UU^{-1}_n)_{n\in\NN}$. Moreover, by \cite[Proposition 3.2]{Guad2012}, 
 $(||||\xx_n-\xx||||_{\UU^{-1}_n})_{n\in\NN}$ is bounded. 
In turn, since $\BB$ and $(J_{\gamma_n \UU_n\AAA})_{n\in\NN}$
are Lipschitzian, and 
$(\forall n\in\NN)\; \xx = J_{\gamma_n\UU_n\AAA}(\xx -\gamma_n\UU_n\BB\xx)$,
we deduce from \eqref{e:Tseng} that 
$(\widetilde{\yy}_n)_{n\in\NN},(\widetilde{\pp}_{n})_{n\in\NN}$, 
and $(\widetilde{\qq}_{n})_{n\in\NN}$ are bounded. 
Therefore, 
\begin{equation}
 \tau = \sup_{n\in\NN}\{||||\xx_n -\widetilde{\yy}_n + \widetilde{\qq}_n -\xx ||||_{\UU_{n}^{-1}},
||||\xx_n-\xx||||_{\UU^{-1}_n}, 1+\eta_n\}
 < +\infty.
\end{equation}
Hence, using \eqref{e:Tseng}, Cauchy-Schwarz for the norms
$(||||\cdot||||_{\UU_{n}^{-1}})_{n\in\NN}$, and \eqref{e:troicho}, we get
\begin{alignat}{2}
(\forall n\in\NN)\;\;
||||\xx_{n+1}-\xx ||||_{\UU_{n}^{-1}}^2 
&= ||||\xx_n -\yy_n +\qq_n -\xx ||||_{\UU_{n}^{-1}}^2\notag\\
&= ||||\widetilde{\qq}_n+ \xx_n -\widetilde{\yy}_n -\xx + \dd_n ||||_{\UU_{n}^{-1}}^2\notag\\
&\leq ||||\widetilde{\qq}_n+ \xx_n -\widetilde{\yy}_n -\xx||||_{\UU^{-1}_n}^2 
+ 2\tau ||||\dd_n||||_{\UU_{n}^{-1}} + ||||\dd_n||||_{\UU_{n}^{-1}}^2\notag\\
&\leq ||||\xx_n-\xx||||_{\UU^{-1}_n}^2 -
\mu^{-1}(1-\gamma^{2}_n \beta^2\mu^2)||||\xx_n-\widetilde{\pp}_n||||^2 + \varepsilon_{1,n},
\end{alignat}
where $(\forall n\in\NN)\; \varepsilon_{1,n} = 
 2\tau ||||\dd_n||||_{\UU_{n}^{-1}} + ||||\dd_n||||_{\UU_{n}^{-1}}^2$.
In turn, for every $n\in\NN$, by \eqref{e:tse} and \cite[Lemma 2.1(i)]{Guad2012},
\begin{alignat}{2}
||||\xx_{n+1}-\xx ||||_{\UU_{n+1}^{-1}}^2 &\leq (1+\eta_n)||||\xx_{n+1}-\xx ||||_{\UU_{n}^{-1}}^2 \notag\\
&\leq ||||\xx_n-\xx||||_{\UU^{-1}_n}^2 -
\mu^{-1}(1-\gamma^{2}_n \beta^2\mu^2)||||\xx_n-\widetilde{\pp}_n||||^2 
+ \tau\varepsilon_{1,n} + \tau^2\eta_n.
\end{alignat}
Since $(\tau\varepsilon_{1,n} + \tau^2\eta_n)_{n\in\NN} \in \ell_{+}^1(\NN)$ by \eqref{e:pabsl},
 it follows from \cite[Lemma 3.1]{plc7} that 
\begin{equation}
\label{e:emyeu}
 \sum_{n\in\NN} ||||\xx_n-\widetilde{\pp}_n||||^2 < +\infty.
\end{equation}

\ref{Tsengi}: 
It follows from \eqref{e:emyeu} and \eqref{e:pabsl} that
\begin{equation}
 \sum_{n\in\NN} ||||\xx_n-\pp_n||||^2  
\leq 2\sum_{n\in\NN} ||||\xx_n-\widetilde{\pp}_n||||^2 
+ 2\sum_{n\in\NN}|||||\pp_n-\widetilde{\pp}_n ||||^2  < +\infty.
\end{equation}
Furthermore, we derive from \eqref{e:pabsl} and \eqref{e:Tseng} that 
\begin{alignat}{2}
 \sum_{n\in\NN} ||||\yy_n -\qq_n ||||^2 &= \sum_{n\in\NN}
  ||||\widetilde{\qq}_n-\widetilde{\yy}_n  +\dd_n ||||^2\notag\\
&= \sum_{n\in\NN}  ||||\widetilde{\pp}_n -\xx_n 
+\gamma_n\UU_n(\BB\xx_n -\BB\widetilde{\pp}_n) +\dd_n ||||^2\notag \\
&\leq 3\Big( \sum_{n\in\NN} ||||\xx_n -\widetilde{\pp}_n||||^2 
+ ||||\gamma_n\UU_n(\BB\xx_n -\BB\widetilde{\pp}_n) ||||^2
 + ||||\dd_n ||||^2\Big)\notag \\
&< +\infty. 
\end{alignat}

\ref{Tsengii}: Let $\xx$ be a weak cluster point of $(\xx_n)_{n\in\NN}$. 
Then there exists a subsequence $(\xx_{k_n})_{n\in\NN}$ that converges weakly to
$\xx$. Therefore $\widetilde{\pp}_{k_n}\weakly \xx$ by \eqref{e:emyeu}. 
Furthermore, it follows from \eqref{e:Tseng} that $\uu_{k_n}\to 0$. 
Hence, since $(\forall n\in\NN )\;(\widetilde{\pp}_{k_n},\uu_{k_n})\in\gra(\AAA+\BB)$, 
we obtain, $\xx\in \zer(\AAA+\BB)$ \cite[Proposition 20.33(ii)]{livre1}. 
Altogether, it follows \cite[Lemma 2.3(ii)]{Guad2012} and \cite[Theorem 3.3]{Guad2012} 
that $\xx_n\weakly \overline{\xx}$ and hence that
$\pp_n\weakly \overline{\xx}$ by \ref{Tsengi}.

\ref{Tsengiiia}: Since $\AAA$ and $\BB$ are maximally monotone and $\dom \BB = \KKK$,
$\AAA+\BB$ is maximally monotone \cite[Corollary 24.4(i)]{livre1},  $\zer(\AAA+\BB)$ is therefore 
closed \cite[Proposition~23.39]{livre1}.
Hence, the claims follow
from \ref{Tsengi}, \eqref{e:oac}, and \cite[Proposition 3.4]{Guad2012}.

\ref{Tsengiiib}: By \ref{Tsengi}, $\xx_n\weakly\overline{\xx}$, and hence \eqref{e:emyeu}
implies that $\widetilde{\pp}_n \weakly\overline{\xx}$. 
Furthermore, it follows from \eqref{e:Tseng} that $\uu_{n}\to 0$. 
Hence, since $(\forall n\in\NN)\;
(\widetilde{\pp}_{n},\uu_{n})\in\gra(\AAA+\BB)$ 
and since $\AAA+\BB$ is demiregular at $\overline{\xx}$, 
by \cite[Definition 2.3]{plc2010}, $\widetilde{\pp}_n\to\overline{\xx}$, 
 and therefore \eqref{e:emyeu}
implies that $\xx_{n}\to\overline{\xx}$.

\ref{Tsengiiic}: 
If $\AAA$ or $\BB$ is uniformly monotone at 
$\overline{\xx}$, then $\AAA +\BB$ is 
uniformly monotone at $\overline{x}$. 
Therefore, the result follows from \cite[Proposition 2.4(i)]{plc2010}.

\ref{Tsengiiid}: Suppose that $\zz\in\inte\zer(\AAA+\BB)$ and fix $\rho\in\RPP$ 
such that $B(\zz;\rho)\subset\zer(\AAA+\BB) $. It follows from 
\eqref{e:oac} and \cite[Proposition 3.2]{Guad2012} that
\begin{equation}
 \varepsilon=\sup_{\xx\in B(\zz;\rho)}\sup_{n\in\NN}
||||\xx_n-\xx||||_{\UU^{-1}_n}\leq(1/\sqrt{\alpha})\big(\sup_{n\in\NN}
||||\xx_n-\zz||||+\sup_{\xx\in B(\zz;\rho)}||||\xx-\zz||||\big)<\pinf
\end{equation}
 and 
from \eqref{e:oac} that 
\begin{align}
\label{e:oaco}
(\forall n\in\NN)(\forall \xx\in B(\zz;\rho))\;
||||\xx_{n+1}-\xx||||^2_{\UU_{n+1}^{-1}}
&\leq||||\xx_n-\xx||||_{\UU^{-1}_n}^2
+2\varepsilon(\varepsilon\eta_n+\varepsilon_n)+(\varepsilon\eta_n+ \varepsilon_n)^2.
\end{align}
Hence, the claim follows from \ref{Tsengi}, \cite[Lemma 2.1]{Guad2012},
and \cite[Proposition 4.3]{Guad2012}.
\end{proof}

\begin{remark} Here are some remarks.
\begin{enumerate}
\item In the case when $(\forall n\in\NN)\; \UU_n= \IId $, the standard 
forward-backward-forward splitting 
algorithm \eqref{e:Tsenga} reduces to algorithm proposed in \cite[Eq. (2.3)]{siop2}, 
which was proposed initially in the error-free setting in \cite{Tseng00}.
\item An alternative variable metric splitting algorithm proposed in \cite{solodov09} 
can be used to find a zero of the sum of a maximally monotone operator $\AAA$
and a Lipschitzian monotone operator $\BB$ in instance 
when $\KKK$ is finite-dimensional. This algorithm uses  a different 
error model and involves more iteration-dependent variables than \eqref{e:Tsenga}.
 \end{enumerate}
\end{remark}

\begin{example}
Let $\boldsymbol{f}\colon\KKK\to\left[-\infty,+\infty\right]$ 
be a proper lower semicontinuous convex function, let $\alpha\in\RPP$, let $\beta\in\RPP$, 
let $\BB\colon\KKK\to\KKK$ be a monotone and $\beta$-Lipschitzian operator, let 
$(\eta_n)_{n\in\NN}\in\ell_+^1(\NN)$, and let $(\UU_n)_{n\in\NN}$ be
a sequence in $\BP_{\alpha}(\KKK)$ that satisfies 
\eqref{e:tse}. Furthermore, let $\xx_0\in\KKK$,
let $\varepsilon\in\left]0,\min\{1,1/(\mu\beta+1)\}\right[$, where 
$\mu$ is defined as in \eqref{e:tse},
let $(\gamma_n)_{n\in\NN}$ be a 
sequence in $[\varepsilon,(1-\varepsilon)/(\beta\mu)]$.
Suppose that the 
variational inequality
\begin{equation}
\label{e:2012-11:10}
\text{find}\quad \xxx\in\KKK\quad\text{such that}\quad
(\forall \yy\in\KKK)\quad\scal{\xxx-\yy}{\BB\xxx}+ \boldsymbol{f}(\xxx)\leq \boldsymbol{f}(\yy)
\end{equation}
admits at least one solution and set
\begin{equation}
\label{e:Tsengavar}
(\forall n\in\NN)\quad
\begin{array}{l}
\left\lfloor
\begin{array}{l}
\yy_n= \xx_n-\gamma_n\UU_n\BB \xx_n\\[1mm]
\pp_n = \arg\underset{\xx\in\KKK}{\min}\big(\boldsymbol{f}(\xx) 
+ \frac{1}{2\gamma_n}||||\xx-\yy_n||||_{\UU_{n}^{-1}}^2\big)\\
\qq_n = \pp_n -\gamma_n \UU_n\BB\pp_n\\
\xx_{n+1} = \xx_n -\yy_n +\qq_n.
\end{array}
\right.\\[2mm]
\end{array}
\end{equation}
Then $(\xx_n)_{n\in\NN}$ converges weakly to a solution 
$\xxx$ to \eqref{e:2012-11:10}.
\end{example}
\begin{proof}
Set $\AAA=\partial \boldsymbol{f}$ and 
$(\forall n\in\NN)\; \aaa_n =0, \bb_n = 0, \cc_n=0$
in Theorem~\ref{l:Tsenglemma}\ref{Tsengii}.
\end{proof}

\section{Monotone inclusions involving Lipschitzian operators}
\label{s:vmpdfbf}
The applications of the forward-backward-forward splitting algorithm 
considered in \cite{siop2,plc6,Tseng00} can be extended to a variable metric 
setting using Theorem \ref{l:Tsenglemma}. As an illustration, we present a variable metric version 
of the algorithm proposed in \cite[Eq. (3.1)]{plc6}.
Recall that
the parallel sum of $A\colon\HH\to 2^{\HH}$ and $B\colon\HH\to 2^{\HH}$ is \cite{livre1} 
\begin{equation}
\label{e:parasum}
A\infconv B=(A^{-1}+ B^{-1})^{-1}.
\end{equation} 

\begin{problem}
\label{CCP}
Let $\HH$ be a real Hilbert space,
let $m$ be a strictly positive integer, let $z\in\HH$,
let $A\colon\HH\to 2^{\HH}$ be maximally monotone operator, 
let $C\colon\HH\to\HH$ be monotone and $\nu_0$-Lipschitzian for some
$\nu_0\in\left]0,+\infty\right[$.
For every $i\in\{1,\ldots, m\}$, let $\GG_i$ be a real Hilbert space, 
let $r_i \in \GG_i$, let $B_i\colon \GG_i\to2^{\GG_i}$ 
be maximally monotone operator, 
let $D_i\colon \GG_i\to2^{\GG_i}$ be monotone and
such that $D_{i}^{-1}$ is $\nu_i$-Lipschitzian for some
$\nu_i\in\left]0,+\infty\right[$,
and let $L_i\colon\HH \to\GG_i$ 
is a nonzero bounded linear operator. 
Suppose that
\begin{equation}
\label{e:fbfcond}
z\in\ran
\bigg(A+\sum_{i=1}^mL^{*}_i\big((B_i\;\vuo\; D_i)(L_i\cdot-r_i)\big) + C\bigg).
\end{equation}
The problem is to solve  the primal inclusion
\begin{equation}\label{primal12:41}
\text{find $\overline{x}\in\HH$ such that}\; 
z\in A\overline{x} +\sum_{i=1}^mL^{*}_i
\big((B_i\;\vuo\; D_i)(L_i\overline{x}-r_i)\big) + C\overline{x} ,
\end{equation}
and the dual inclusion 
\begin{equation}\label{dual12:41}
\text{find $\overline{v}_1 \in \GG_1,\ldots, \overline{v}_m \in \GG_m$ such that 
 $(\exists x\in\HH)$ }
\begin{cases}
z -\sum_{i = 1}^m L_{i}^*\overline{v}_i  \in Ax +Cx,\\
(\forall i\in\{1,\ldots,m\})\;\overline{v}_i\in (B_i\;\vuo\; D_i)(L_ix-r_i).
\end{cases}
\end{equation}
\end{problem}
As shown in \cite{plc6}, Problem \ref{CCP} covers 
a wide class of problems in nonlinear analysis and
convex  optimization problems.
However, the algorithm in \cite[Theorem 3.1]{plc6} is studied in the context of a fixed 
metric. The following result extends this result to a variable metric setting.

\begin{corollary}
\label{t:Sv2}
Let $\alpha$ be in $\left]0,\pinf\right[$, 
let $(\eta_{0,n})_{n\in\NN}$ be a sequence in $\ell_{+}^1(\NN)$,
let $(U_n)_{n\in\NN}$ be a sequence in $\BP_{\alpha}(\HH)$,
and for every $i\in\{1,\ldots,m\}$, 
let $(\eta_{i,n})_{n\in\NN}$ be a sequence in $\ell_{+}^1(\NN)$,
let $(U_{i,n})_{n\in\NN}$ be a sequence in $\BP_{\alpha}(\GG_i)$ such that 
$\mu=\sup_{n\in\NN}\{ \|U_n\|, \|U_{1,n}\|,\ldots, \|U_{m,n}\|\} < +\infty$ and 
\begin{alignat}{2}
(\forall n\in\NN)\quad 
(1+\eta_{0,n})U_{n+1} \succcurlyeq U_n,\quad &\text{and} 
\quad (\forall i\in \{1,\ldots,m\})\quad
(1+\eta_{i,n})U_{i,n+1} \succcurlyeq U_{i,n}.
\end{alignat}
Let $(a_{1,n})_{n\in\NN}, (b_{1,n})_{n\in\NN}$, and $(c_{1,n})_{n\in\NN}$
be absolutely summable sequences in $\HH$, and for every $i\in\{1,\ldots,m\}$, 
let $(a_{2,i,n})_{n\in\NN}, (b_{2,i,n})_{n\in\NN}$, and $(c_{2,i,n})_{n\in\NN}$ 
be absolutely summable sequences 
in $\GG_i$. Furthermore, set 
\begin{equation}
 \beta = \max\{\nu_0,\nu_1,\ldots, \nu_m\} + \sqrt{\sum_{i=1}^m\|L_i\|^2},
\end{equation}
let $x_0\in\HH$, let $(v_{1,0}, \ldots, v_{m,0}) \in \GG_1\oplus\ldots\oplus\GG_m$, 
let $\varepsilon \in \left]0, 1/(1+\beta\mu)\right[$, let $(\gamma_n)_{n\in\NN}$ be a sequence in 
$[\varepsilon, (1-\varepsilon)/(\beta\mu)]$. Set 
\begin{equation}
\label{e:sva2}
(\forall n\in\NN)\quad 
\begin{array}{l}
\left\lfloor
\begin{array}{l}
y_{1,n} = x_n - \gamma_n U_n\big(Cx_n+ \sum_{i=1}^{m}L_{i}^*v_{i,n}+ a_{1,n}\big)\\
p_{1,n}=J_{\gamma_n U_{n}A}(y_{1,n} + \gamma_nU_nz) + b_{1,n}\\
\operatorname{for}\  i =1,\ldots, m\\
\left\lfloor
\begin{array}{l}
y_{2,i,n} = v_{i,n} + \gamma_nU_{i,n}\big(L_ix_n - D_{i}^{-1}v_{i,n} + a_{2,i,n}\big)\\
p_{2,i,n} = J_{\gamma_n U_{i,n}B_{i}^{-1}}(y_{2,i,n} -\gamma_nU_{i,n}r_i) + b_{2,i,n}\\
q_{2,i,n} = p_{2,i,n} + \gamma_nU_{i,n}\big(L_ip_{1,n} -D_{i}^{-1}p_{2,i,n} + c_{2,i,n}\big)\\
v_{i,n+1} = v_{i,n} - y_{2,i,n} + q_{2,i,n} 
\end{array}
\right.\\[2mm]
q_{1,n} = p_{1,n} -\gamma_nU_n\big(Cp_{1,n} + \sum_{i=1}^m L^{*}_ip_{2,i,n} + c_{1,n}\big)\\
x_{n+1} = x_n - y_{1,n} + q_{1,n}.
\end{array}
\right.\\[2mm]
\end{array}
\end{equation}
Then the following hold. 
\begin{enumerate}
\item \label{t:Sv2i} $\sum_{n\in\NN}\|x_n - p_{1,n}\|^2 < +\infty$ and 
$(\forall i\in \{1,\ldots,m\})\; \sum_{n\in\NN}\|v_{i,n} - p_{2,i,n}\|^2 < +\infty$. 
 \item\label{t:Sv2ii}  
There exist a solution $\overline{x}$ to \eqref{primal12:41} and a solution  
$(\overline{v}_1,\ldots, \overline{v}_m)$ to \eqref{dual12:41} such that the following hold.
\begin{enumerate}
 \item \label{t:Sv2iia}
$x_n\weakly \overline{x}$ and $p_{1,n} \weakly \overline{x}$.
\item \label{t:Sv2iib}  $(\forall i\in \{1,\ldots,m\})\; v_{i,n} \weakly \overline{v}_i$ and 
$p_{2,i,n} \weakly \overline{v}_i$.
\item \label{t:Sv2iic} Suppose that $A$ or $C$ is uniformly monotone at $\overline{x}$, then 
$x_n\to\overline{x}$ and $p_{1,n}\to\overline{x}$.
\item \label{t:Sv2iid} 
Suppose that $B^{-1}_j$ or $D_{j}^{-1}$ is uniformly monotone at $\overline{v}_j$, 
for some $j\in\{1,\ldots,m\}$, then 
$v_{j,n}\to\overline{v}_j$ and $p_{2,j,n}\to\overline{v}_j$.
\end{enumerate}
\end{enumerate}
\end{corollary}
\begin{proof} All sequences generated by algorithm \eqref{e:sva2} are well defined 
by \cite[Lemma 3.7]{Varm12}.
We define $\KKK=\HH\oplus\GG_1\oplus\cdots\oplus\GG_m$ the Hilbert direct 
sum of the Hilbert spaces $\HH$ and $(\GG_i)_{1\leq i\leq m}$,
the scalar product and the associated 
norm of $\KKK$ respectively defined by
\begin{equation}
\label{e:palawan-mai2008b-}
\psscal{\cdot}{\cdot}
\colon((x,\boldsymbol{v}),(y,\boldsymbol{w}))\mapsto
\scal{x}{y}+ \sum_{i=1}^m\scal{v_i}{w_i}
\quad\text{and}\quad||||\cdot||||\colon
(x,\boldsymbol{v})\mapsto\sqrt{\|x\|^2+\sum_{i=1}^m\|v_i\|^2},
\end{equation}
where $\vv = (v_1,\ldots,v_m)$ and $\ww = (w_1,\ldots,w_m)$ are generic elements in 
$\GG_1\oplus\cdots\oplus\GG_m$.
 Set
\begin{equation}
 \label{e:maximal1}
\begin{cases}
\AAA\colon\KKK\to 2^{\KKK}\colon
(x,v_1,\ldots,v_m)\mapsto (-z+ Ax)
\times(r_1 + B_{1}^{-1}v_1)\times\ldots\times(r_m + B^{-1}_{m}v_m)\\
\BB\colon\KKK\to \KKK\colon
(x,v_1,\ldots,v_m)\mapsto
\bigg(Cx+\sum_{i=1}^mL_{i}^*v_i,D_{1}^{-1}v_1-L_1x,\ldots,D_{m}^{-1}v_m-L_mx\bigg)\\
(\forall n\in\NN)\quad 
\UU_n\colon\KKK\to \KKK\colon
(x,v_1,\ldots,v_m)\mapsto \big(U_nx, U_{1,n}v_1,\ldots U_{m,n}v_m\big).\\
\end{cases}
\end{equation}
Since $\AAA$ is maximally monotone \cite[Propositions 20.22 and 20.23]{livre1},
$\BB$ is monotone and $\beta$-Lipschitzian \cite[Eq. (3.10)]{plc6} 
with $\dom\BB=\KKK$, $\AAA+\BB$ is maximally monotone~\cite[Corollary~24.24(i)]{livre1}. 
Now set 
$(\forall n\in\NN)\; \eta_n= \max\{\eta_{0,n},\eta_{1,n},\ldots, \eta_{m,n}\}$. 
Then $(\eta_n)_{n\in\NN}\in\ell_{+}^1(\NN)$. Moreover, we derive from our assumptions 
on the sequences $(U_n)_{n\in\NN}$ and $(U_{1,n})_{n\in\NN}, \ldots,(U_{m,n})_{n\in\NN}$ that
\begin{equation}
\label{e:tsa1}
 \mu =\sup_{n\in\NN} \|\UU_n\| < +\infty
\quad \text{and}\quad (1+\eta_n)\UU_{n+1} \succcurlyeq \UU_n \in \BP_{\alpha}(\KKK).
\end{equation}
In addition, \cite[Propositions 23.15(ii) and 23.16]{livre1} yield
$(\forall \gamma \in \left]0,\pinf\right[)(\forall n\in\NN)(\forall (x,v_1,\ldots, v_m) \in\KKK)$
\begin{alignat}{2}
\label{e:coco} 
J_{\gamma \UU_n \AAA}(x,v_1,\ldots,v_m) = \Big(J_{\gamma U_n A }(x+\gamma U_nz), 
\big(J_{\gamma U_{i,n}B^{-1}_i}(v_i-\gamma U_{i,n}r_i)\big)_{1\leq i\leq m}\Big).
\end{alignat}
It is shown in~\cite[Eq.~(3.12)]{plc6} and \cite[Eq.~(3.13)]{plc6} 
that under the condition~\eqref{e:fbfcond},
$\zer(\AAA + \BB )\neq\emp$. Moreover, \cite[\rm Eq.~(3.21)]{plc6} and 
~\cite[\rm Eq.~(3.22)]{plc6} yield
\begin{equation}
(\overline{x},\overline{v}_1,\ldots, \overline{v}_m) 
\in\zer(\AAA+ \BB)\Rightarrow \overline{x}\;\text{solves \eqref{primal12:41}}\;
\text{and}\;
(\overline{v}_1,\ldots, \overline{v}_m)\; \text{solves \eqref{dual12:41}}.
\end{equation}
Let us next set 
\begin{equation}
\label{e:e:cucucu}
(\forall n\in\NN)\quad
 \begin{cases}
 \xx_n = (x_n,v_{1,n},\ldots,v_{m,n})\\
\yy_n = (y_{1,n},y_{2,1,n},\ldots, y_{2,m,n})\\
\pp_n = (p_{1,n},p_{2,1,n},\ldots, p_{2,m,n})\\
\qq_n = (q_{1,n},q_{2,1,n},\ldots, q_{2,m,n})
 \end{cases}
\quad \text{and}\quad 
\begin{cases}
 \aaa_n =  (a_{1,n},a_{2,1,n},\ldots, a_{2,m,n})\\
\bb_n = (b_{1,n},b_{2,1,n},\ldots, b_{2,m,n})\\
\cc_n = (c_{1,n},c_{2,1,n},\ldots, c_{2,m,n}).
\end{cases}
\end{equation}
Then our assumptions imply that 
\begin{equation}
 \sum_{n\in\NN}||||\aaa_n|||| < \infty,\quad \sum_{n\in\NN}||||\bb_n|||| < \infty,\quad 
\text{and}\quad  \sum_{n\in\NN}||||\cc_n|||| < \infty.
\end{equation}
Furthermore, it follows from the definition of $\BB$, \eqref{e:coco}, and  \eqref{e:e:cucucu} 
that \eqref{e:sva2} can be rewritten in $\KKK$ as 
 \begin{equation}
\label{e:Tsengb}
(\forall n\in\NN)\quad
\begin{array}{l}
\left\lfloor
\begin{array}{l}
\yy_n= \xx_n-\gamma_n\UU_n(\BB\xx_n+\aaa_n)\\[1mm]
\pp_n = J_{\gamma_n \UU_n \AAA}\yy_n + \bb_n\\
\qq_n = \pp_n -\gamma_n\UU_n(\BB\pp_n+\cc_n)\\
\xx_{n+1} = \xx_n -\yy_n +\qq_n,
\end{array}
\right.\\[2mm]
\end{array}
\end{equation}
which is \eqref{e:Tsenga}. Moreover,  
every specific conditions in Theorem \ref{l:Tsenglemma} are satisfied.

\ref{t:Sv2i}: By Theorem \ref{l:Tsenglemma}\ref{Tsengi}, 
$\sum_{n\in\NN}||||\xx_n-\pp_{n} ||||^2 < \infty$.

\ref{t:Sv2iia}\&\ref{t:Sv2iib}: These assertions follow from Theorem \ref{l:Tsenglemma}\ref{Tsengii}.

\ref{t:Sv2iic}: Theorem \ref{l:Tsenglemma}\ref{Tsengii} shows that 
$(\overline{x},\overline{v}_1, \ldots, \overline{v}_m)\in\zer(\AAA+\BB)$. Hence, it follows
from~\cite[Eq~(3.19)]{plc6} that
 $(\overline{x},\overline{v}_1, \ldots, \overline{v}_m)$ satisfies the inclusions
\begin{equation}
\label{e:dualsol}
\begin{cases}
-\sum_{i = 1}^m L_{i}^*\overline{v}_{i} - C\overline{x}\in -z+A\overline{x} \\
(\forall i\in\{1,\ldots,m\})\;L_i\overline{x}- D^{-1}_i\overline{v}_{i}\in r_i+ B^{-1}_i\overline{v}_{i}.
\end{cases}
\end{equation}
For every $n\in\NN$ and every $i\in\{1,\ldots,m\}$, set
\begin{equation}
\label{e:sva21}
 \begin{cases}
  \widetilde{y}_{1,n} = x_n - \gamma_n U_n\big(Cx_n+ \sum_{i=1}^{m}L_{i}^*v_{i,n}\big)\\
\widetilde{p}_{1,n}=J_{\gamma_n U_{n}A}(\widetilde{y}_{1,n} + \gamma_nU_nz)\\
 \end{cases}
\quad \text{and}\quad
\begin{cases}
\widetilde{y}_{2,i,n} = v_{i,n} + \gamma_nU_{i,n}\big(L_ix_n - D_{i}^{-1}v_{i,n}\big)\\
\widetilde{p}_{2,i,n} = J_{\gamma_n U_{i,n}B_{i}^{-1}}(\widetilde{y}_{2,i,n} -\gamma_nU_{i,n}r_i).\\
\end{cases}
\end{equation}
Then, using \cite[Lemma 3.7]{Varm12}, we get 
\begin{equation}
 \widetilde{p}_{1,n} - p_{1,n}\to 0 \quad \text{and}\quad(\forall i\in \{1,\ldots,m\})
\quad \widetilde{p}_{2,i,n}-p_{2,i,n} \to 0,
\end{equation}
in turn, by \ref{t:Sv2i},\ref{t:Sv2iia}, and \ref{t:Sv2iib}, we obtain
\begin{equation}
\label{e:hoaquangoc}
\widetilde{p}_{1,n} - x_n\to 0, \quad \widetilde{p}_{1,n} \weakly \overline{x},
\quad \text{and}\quad(\forall i\in\{1,\ldots,m\})\quad
\widetilde{p}_{2,i,n}-v_{i,n} \to 0,\quad \widetilde{p}_{2,i,n}\weakly \overline{v}_i.
\end{equation}
Furthermore, we derive from
\eqref{e:sva21} that 
\begin{equation}
\label{e:dualsola}
(\forall n\in\NN)\quad
 \begin{cases}
 \gamma^{-1}_nU_{n}^{-1}(x_n-\widetilde{p}_{1,n})
-\sum_{i = 1}^m L_{i}^*v_{i,n} - Cx_n \in -z + A\widetilde{p}_{1,n}\\
(\forall i\in \{1,\ldots,m\})\;
\gamma^{-1}_nU_{i,n}^{-1}(v_{i,n}-\widetilde{p}_{2,i,n}) + L_ix_n  
- D^{-1}_i v_{i,n} \in r_i + B_{i}^{-1}\widetilde{p}_{2,i,n}.\\
 \end{cases}
\end{equation}
Since $A$ is uniformly monotone at $\overline{x}$, 
using \eqref{e:dualsol} and \eqref{e:dualsola},
there exists an increasing function
$\phi_A\colon\left[0,+\infty\right[\to\left[0,+\infty\right]$ 
vanishing only at $0$ such that, for every $n\in\NN$,  
\begin{alignat}{2}
 \phi_A(\|\widetilde{p}_{1,n} -\overline{x}\|) &\leqslant 
\scal{\widetilde{p}_{1,n}-\overline{x}}{\gamma^{-1}_nU^{-1}_n(x_{n}-\widetilde{p}_{1,n}) 
- \sum_{i = 1}^m(L_{i}^*v_{i,n} - L_{i}^*\overline{v}_{i}) - (Cx_n-C\bar{x})}\notag\\
&= \scal{\widetilde{p}_{1,n}-\overline{x}}{\gamma^{-1}_nU^{-1}_n(x_{n}-\widetilde{p}_{1,n})}- 
\sum_{i = 1}^m\scal{\widetilde{p}_{1,n}
-\overline{x}}{L_{i}^*v_{i,n} - L_{i}^*\overline{v}_{i}}-\chi_n\label{cz1},
\end{alignat}
where we denote $\big(\forall n\in\NN\big)\; \chi_n = \scal{\widetilde{p}_{1,n} -\bar{x}}{Cx_n -C\bar{x}}$.
Since $(B_{i}^{-1})_{1\leq i\leq m}$ are monotone, for every
$i\in\{1,\ldots,m\}$, we obtain
\begin{alignat}{2}
(\forall n\in\NN)\quad
0 &\leqslant \scal{\widetilde{p}_{2,i,n} - \overline{v}_{i}}{L_ix_{n} 
+\gamma^{-1}_nU^{-1}_{i,n}(v_{i,n}-\widetilde{p}_{2,i,n}) -L_i\overline{x} - (D_{i}^{-1}v_{i,n} - D_{i}^{-1}
\bar{v}_{i})}\notag\\
&= \scal{\widetilde{p}_{2,i,n} - \overline{v}_i}{L_i(x_{n} -\overline{x}) 
+\gamma^{-1}_nU^{-1}_{i,n}(v_{i,n}-\widetilde{p}_{2,i,n})} 
- \beta_{i,n},
\label{cz2}
\end{alignat}
where $\big(\forall n\in\NN\big)\;
 \beta_{i,n} = \scal{\widetilde{p}_{2,i,n} - \overline{v}_{i}}{D_{i}^{-1}v_{i,n} - D_{i}^{-1}
\bar{v}_{i}}$.
Now, adding~\eqref{cz2} from $i = 1$ to $i = m$ and~\eqref{cz1}, we obtain,
for every $n\in\NN$,
\begin{alignat}{2}\label{cz3}
 \phi_A(\|\widetilde{p}_{1,n} -\overline{x}\|) 
&\leq\scal{\widetilde{p}_{1,n}-\overline{x}}{\gamma^{-1}_nU^{-1}_n(x_{n}-\widetilde{p}_{1,n})}
+\scal{\widetilde{p}_{1,n}-\overline{x}}{ \sum_{i=1}^m L_{i}^*(\widetilde{p}_{2,i,n}-v_{i,n})}
\notag \\
&\quad+\sum_{i = 1}^m \scal{\widetilde{p}_{2,i,n} - \overline{v}_i}{L_i(x_{n} -\widetilde{p}_{1,n}) 
+\gamma^{-1}_nU^{-1}_{i,n}(v_{i,n}-\widetilde{p}_{2,i,n})}
-\chi_n - \sum_{i = 1}^{m}\beta_{i,n}.
\end{alignat}
For every $n\in\NN$ and every $ i\in \{1,\ldots,m\}$,
we expand $\chi_n$ and $\beta_{i,n}$ as
\begin{equation}
\label{e:expa}
\begin{cases}
\chi_n = \scal{x_n-\overline{x}}{ Cx_n- C\overline{x}}
 + \scal{\widetilde{p}_{1,n} -x_n}{ Cx_n- C\overline{x}},\\
\;\beta_{i,n} =
\scal{v_{i,n}-\overline{v}_i}{ D_{i}^{-1}v_{i,n}- D_{i}^{-1}\overline{v}_i}
 + \scal{\widetilde{p}_{2,i,n} - v_{i,n}}{ D_{i}^{-1}v_{i,n}- D_{i}^{-1}\overline{v}_i}.
\end{cases}
\end{equation}
By  monotonicity of $C$ and $(D_{i}^{-1})_{1\leq i\leq m}$,
\begin{equation}
(\forall n\in\NN)\quad
 \begin{cases}
\scal{x_n-\overline{x}}{ Cx_n- C\overline{x}}\geq 0,\\
(\forall i\in \{1,\ldots,m\})\; 
\scal{v_{i,n}-\overline{v}_i}{ D_{i}^{-1}v_{i,n}- D_{i}^{-1}\overline{v}_i}\geq 0.
 \end{cases}
\end{equation}
Therefore, for every $n\in\NN$, we derive from \eqref{e:expa} and \eqref{cz3} that
\begin{alignat}{2}
\label{e:concao}
 \phi_A(\|\widetilde{p}_{1,n} -\overline{x}\|)&\leq
 \phi_A(\|\widetilde{p}_{1,n} -\overline{x}\|) + \scal{x_n-\overline{x}}{ Cx_n- C\overline{x}}
+ \sum_{i=1}^m \scal{v_{i,n}-\overline{v}_i}{ D_{i}^{-1}v_{i,n}- D_{i}^{-1}\overline{v}_i}\notag\\
&\leq\scal{\widetilde{p}_{1,n}-\overline{x}}{\gamma^{-1}_n U^{-1}_n(x_{n}-\widetilde{p}_{1,n})}
+\scal{\widetilde{p}_{1,n}-\overline{x}}{ \sum_{i=1}^m L_{i}^*(\widetilde{p}_{2,i,n}-v_{i,n})}
\notag \\
&\quad+\sum_{i = 1}^m \scal{\widetilde{p}_{2,i,n} - \overline{v}_i}{ L_i(x_{n} -\widetilde{p}_{1,n}) 
+\gamma^{-1}_nU^{-1}_{i,n}(v_{i,n}-\widetilde{p}_{2,i,n})}\notag\\
&\quad - \scal{\widetilde{p}_{1,n} -x_n}{ Cx_n- C\overline{x}} 
-\sum_{i=1}^m\scal{\widetilde{p}_{2,i,n} - v_{i,n}}{ D_{i}^{-1}v_{i,n}- D_{i}^{-1}\overline{v}_i}.
\end{alignat}
We set
\begin{equation}
 \zeta = \max_{1\leq i\leq m} \sup_{n\in\NN}\{\|x_n -\overline{x}||,\|\widetilde{p}_{1,n} -\overline{x}\|, 
\|v_{i,n}-\overline{v}_i\|, \|\widetilde{p}_{2,i,n} - \overline{v}_i\| \}. 
\end{equation}
Then it follows from \ref{t:Sv2iia}, \ref{t:Sv2iib}, and \eqref{e:hoaquangoc}
that $\zeta < \infty$, and  from \cite[Lemma 2.1(ii)]{Guad2012} that
$(\forall n\in\NN)\;\|\gamma^{-1}_nU^{-1}_n\| \leq (\varepsilon\alpha)^{-1}$ and
$(\forall i \in\{1,\ldots,m\})\; \| \gamma^{-1}_nU_{i,n}^{-1}\| \leq (\varepsilon\alpha)^{-1}$. 
Therefore, using the Cauchy-Schwarz inequality, and the Lipschitzianity 
of $C$ and $(D^{-1}_i)_{1\leq i\leq m}$,
we derive from \eqref{e:concao} that
\begin{alignat}{2}\label{cz3e}
 \phi_A(\|\widetilde{p}_{1,n} -\overline{x}\|) 
&\leq (\varepsilon\alpha)^{-1}\zeta\|x_{n}-\widetilde{p}_{1,n}\|
+\zeta\sum_{i = 1}^m\big(\|L_i\|\; \|x_{n} -\widetilde{p}_{1,n}\| 
+ (\varepsilon\alpha)^{-1}\| v_{i,n}-\widetilde{p}_{2,i,n}\|\big)\notag \\
&\quad+\zeta\bigg(\sum_{i=1}^m \|L_{i}^*\| \|\widetilde{p}_{2,i,n}-v_{i,n}\|
+\nu_0\|\widetilde{p}_{1,n} -x_n\| 
+ \sum_{i=1}^m \nu_i \|\widetilde{p}_{2,i,n} - v_{i,n}\|\bigg) \notag\\
&\to 0.
\end{alignat}
We deduce from~\eqref{cz3e} and \eqref{e:hoaquangoc} 
that $\phi_A(\|\widetilde{p}_{1,n} -\overline{x}\|)\to 0$, which implies that 
$\widetilde{p}_{1,n}\to \overline{x}$. In turn,
$x_n \to \overline{x}$ and $p_n\to \overline{x}$.
Likewise, if $C$
is uniformly monotone at $\overline{x}$, 
there exists an increasing function $\phi_C\colon\left[0,+\infty\right[\to\left[0,+\infty\right]$ that vanishes
only at $0$ such that
\begin{alignat}{2}\label{cz3ef}
 \phi_C(\|x_{n} -\overline{x}\|) 
&\leq (\varepsilon\alpha)^{-1}\zeta\|x_{n}-\widetilde{p}_{1,n}\|
+\zeta\sum_{i = 1}^m\big(\|L_i\|\; \|x_{n} -\widetilde{p}_{1,n}\| 
+ (\varepsilon\alpha)^{-1}\| v_{i,n}-\widetilde{p}_{2,i,n}\|\big)\notag \\
&\quad+\zeta\bigg(\sum_{i=1}^m \|L_{i}^*\| \|\widetilde{p}_{2,i,n}-v_{i,n}\|
+\nu_0\|\widetilde{p}_{1,n} -x_n\| + \sum_{i=1}^m \nu_i \|\widetilde{p}_{2,i,n} - v_{i,n}\|\bigg) \notag\\
&\to 0,
\end{alignat}
in turn,  $x_n \to \overline{x}$ and $p_n\to \overline{x}$.

\ref{t:Sv2iid}: Proceeding
as in the proof of \ref{t:Sv2iic}, we obtain the conclusions.
\end{proof}

\noindent{{\bfseries Acknowledgement.}}
I thank Professor Patrick L. Combettes for bringing this problem 
to my attention and for helpful discussions.

\end{document}